\documentclass{amsart}
\usepackage[english]{babel}
\usepackage{amsthm,amsmath}
\usepackage{inputenc}
\usepackage{amssymb}
\usepackage{graphicx}
\usepackage{color}
\newtheorem{thm}{Theorem}[section]
\newtheorem{cor}[thm]{Corollary}
\newtheorem{lem}[thm]{Lemma}
\newtheorem{prop}[thm]{Proposition}
\newtheorem{defn}[thm]{Definition}

\numberwithin{equation}{section}

\begin{document}
\title[Leibniz algebras whose semisimple part is related to $sl_2$.]
{Leibniz algebras whose semisimple part is related to $sl_2$.}

\author{L.M. Camacho, S. G\'{o}mez-Vidal, B.A. Omirov and I.A. Karimjanov}
\address{[L.M. Camacho --- S. G\'{o}mez-Vidal] Dpto. Matem\'{a}tica Aplicada I.
Universidad de Sevilla. Avda. Reina Mercedes, s/n. 41012 Sevilla.
(Spain)} \email{lcamacho@us.es --- samuel.gomezvidal@gmail.com}
\address{[B.A. Omirov --- I.A. Karimjanov] Institute of Mathematics and Information Technologies
of Academy of Uzbekistan, 29, Do'rmon yo'li street., 100125, Tashkent (Uzbekistan)}
\email{omirovb@mail.ru ---iqboli@gmail.com}

%
\thanks{This work has been funded by Mathematics Institute and V Research Plan of Sevilla University by IMU/CDC-program.}%

\begin{abstract}
In this paper we identify the structure of complex finite-dimensional Leibniz algebras with associated Lie algebras $sl_2^1\oplus sl_2^2\oplus \dots \oplus sl_2^s\oplus R,$ where $R$ is a solvable radical. The classifications of such Leibniz algebras in the cases $dim R=2, 3$ and $dim I\neq 3$ have been obtained. Moreover, we classify Leibniz algebras with $L/I\cong sl_2^1\oplus sl_2^2$ and some conditions on ideal $I=id<[x,x] \ | \ x\in L>.$
\end{abstract}

\maketitle \textbf{Mathematics Subject Classification 2010}:
17A32, 17A60, 17B10, 17B20.


\section{Introduction.}

The notion of Leibniz algebra was first introduced by Loday in
\cite{loday}, \cite{lodpir} as a non-antisymmetric generalization
of Lie algebra. During the last $20$ years the theory of Leibniz
algebras has been actively studied and many results of the theory
of Lie algebras have been extended to Leibniz algebras. A lot of
papers have  been devoted to the description of
finite-dimensional nilpotent Leibniz algebras \cite{AOR1},
\cite{Ayu} so far. However, just a few works are related to the
semisimple part of Leibniz algebras \cite{Sam}, \cite{Bar},
\cite{Rakh1}.

We know that an arbitrary Lie algebra can be
decomposed into a semidirect sum of the solvable radical and its semisimple subalgebra (Levi's Theorem \cite{Jac}).
According to the Cartan-Killing theory, a semisimple Lie algebra can be represented as a direct sum of simple ideals,
which are completely classified \cite{Jac}.

In a recent study, Barnes has proved an analogue of Levi's Theorem for the case of Leibniz algebras \cite{Bar}.
Namely, a Leibniz algebra is decomposed into a semidirect sum of its solvable radical and a semisimple Lie algebra.

The inherent properties of non-Lie Leibniz algebras imply that the
subspace spanned by squares of elements of the algebra is a
non-trivial ideal (denoted by $I$). Moreover, the ideal
$I$ is abelian and hence, it is contained in the solvable radical.
Thanks to result of Barnes in order to describe Leibniz algebras it is enough to investigate the
relationship between products of a semisimple Lie algebra and the
radical (see \cite{Sam}, \cite{Rakh} and  \cite{Rakh1}).

The present work devoted to describing the structure of Leibniz
algebras with the associated Lie algebras isomorphic to $sl_2^1\oplus
sl_2^2\oplus \dots \oplus sl_2^s\oplus R$ with $I$ an irreducible right $sl_2^k$-module for some $k.$ Since the description of such Leibniz algebras is very complicated,
we have focused on Leibniz algebras with semisimple part $sl_2^1\oplus sl_2^2$ under some conditions on the ideal $I$.

In order to achieve our goal we organize the paper as follows. In
Section 2, we give some necessary notions and preliminary results
about Leibniz algebras with associated Lie algebra $sl_2\dot{+}
R.$ Section 3 is devoted to the study of the structure of Leibniz
algebras, whose semisimple part is a direct sum of copies of
$sl_2$ algebras and with some conditions on the ideal $I.$ In
Section 4, we classify Leibniz algebras whose semisimple part is a
direct sum $sl_2^1\oplus sl_2^2\ $ and $I$ is decomposed into a
direct sum of two irreducible modules $I_{1,1}, I_{1,2}$ over
$sl_2^1$ such that $dimI_{1,1}=dimI_{1,2}.$ The description of the structure of Leibniz algebras with associated Lie algebra $sl_n$ and $dim I=1,2$ is obtained in Section 5.

Throughout the work, vector spaces and algebras are
finite-dimensional over the field of complex numbers. Moreover, in
the table of multiplication of an algebra the omitted products are
assumed to be zero. We shall use the following symbols: $+$,
$\oplus$ and $\dot{+}$ for notations of the direct sum of the
vector spaces, the  direct and semidirect sums of algebras,
respectively.

\section{Preliminaries}

In this section we give some necessary definitions and preliminary results.

\begin{defn} \cite{loday} An algebra $(L,[\cdot,\cdot])$ over a field $F$ is called a
Leibniz algebra if for any $x,y,z\in L$ the so-called Leibniz identity
$$[x,[y,z]]=[[x,y],z] - [[x,z],y]$$ holds.
\end{defn}

Let $L$ be a Leibniz algebra and let $I=ideal< [x,x]\  | \ x\in L
>$ be an ideal of $L$ generated by all squares. The natural
epimorphism $\varphi  : L \rightarrow   L/I$ determines the
associated Lie algebra $L/I$ of the Leibniz algebra $L.$ It is
clear that ideal $I$ is the minimal ideal with respect to the
property that the quotient algebra by this ideal is a Lie algebra.

It is noted that in \cite{Bar} the ideal $I$ coincides with the
space spanned by squares of elements of an algebra.


According to \cite{Jac} there exists a unique (up to isomorphism) $3$-dimensional simple Lie algebra with the following table of multiplication:
$$sl_2: \quad  [e,h]=-[h,e]=2e, \qquad [h,f]=-[f,h]=2f, \qquad [e,f]=-[f,e]=h,$$
The basis $\{e,f,h\}$ is called the {\it canonical basis}.

In \cite{Rakh}, Leibniz algebras whose quotient Lie
algebra is isomorphic to $sl_2$ are described. Let us present a
Leibniz algebra $L$ with the table of multiplication in a basis
$\{e, f, h, x_{0}^1, \dots, x_{t_1}^1, $ $x_0^2, \dots, x_{t_2}^2, \dots, x_0^p, \dots,$ $ x_{t_p}^p\}$  and the
quotient algebra $L/I$ is $sl_2$:
$$\begin{array}{lll}
\, [e,h]=-[h,e]=2e, & [h,f]=-[f,h]=2f, &[e,f]=-[f,e]=h, \\
\, [x_k^{j},h]=(t_j-2k)x_k^{j}, & 0 \leq k \leq t_j ,&\\
\, [x_k^{j},f]=x_{k+1}^{j},  & 0 \leq k \leq t_j-1, & \\
\, [x_k^{j},e]=-k(t_j+1-k)x_{k-1}^{j}, & 1 \leq k \leq t_j. &\\
 \end{array}$$
where $L=sl_2 + I_1+I_2+ \dots +I_p$ and $I_j=\langle
x_1^j,\dots,x_{t_j}^j \rangle, \, 1\leq j \leq p.$

The last three types of products of the above table of
multiplication are characterized as an irreducible $sl_2$-module
with the canonical basis of $sl_2$ \cite{Jac}.

Now we give the notion of semisimplicity for Leibniz algebras.
\begin{defn} \cite{Sam} A Leibniz algebra is called semisimple if its
maximal solvable ideal is equal to $I$. \end{defn}

Since in the Lie
algebras case the ideal $I$ is equal to zero, this definition also
agrees with the definition of semisimple Lie algebra.

Although Levi's Theorem is proved for the left Leibniz algebras
\cite{Bar}, it is also true for right Leibniz algebras (here we
consider the right Leibniz algebras).

\begin{thm}\cite{Bar} (Levi's Theorem).\label{t2}
Let $L$ be a finite dimensional Leibniz algebra over a field of
characteristic zero and $R$ be its solvable radical. Then there
exists a semisimple subalgebra $S$ of $L$, such that
$L=S\dot{+}R.$
\end{thm}

An algebra $L$ is called \emph{simple} if it only ideals are
$\{0\}, I, L$ and $L^2\neq I$. From the proof of
Theorem \ref{t2}, it is not difficult to see that $S$ is a
semisimple Lie algebra. Therefore, we have that a simple Leibniz
algebra is a semidirect sum of simple Lie algebra $S$ and the
irreducible right module $I$ over $S$, i.e. $L=S\dot{+}I.$ Hence,
we get the description of the simple Leibniz algebras in terms of
simple Lie algebras and ideal $I$.

\begin{defn} \cite{Jac} A non-zero module $M$ over a Lie algebra whose only submodules are the
module itself and zero module is called {\it irreducible module}. A
non-zero module $M$ which admits decomposition into a direct sum of irreducible modules is
said to be {\it completely reducible}.
\end{defn}

Further, we shall use the following Weyl's semisimplicity theorem.

\begin{thm}\cite{Jac} \label{thm25} Let G be a semisimple Lie algebra over a field of characteristic zero.
Then every finite dimensional module over $G$ is completely
reducible.
\end{thm}

Now we present results on classification of Leibniz
algebras with the conditions $L/I\cong sl_2\oplus R, \ dimR=2, 3,$ where $I$ is an irreducible right module over $sl_2\ (dim I\neq3)$.

\begin{thm}\label{thm01} \cite{Lisa}
 Let $L$ be a Leibniz algebra whose quotient $L/I\cong sl_2\oplus R$, where $R$ is a two-dimensional solvable ideal and  $I$ is a right irreducible module over $sl_2\ (dim I\neq3)$. Then there exists a basis $\{e, h, f,
x_0, x_1, \dots, x_m, y_1, {y_2}\}$ of the algebra $L$ such that the table of multiplication in $L$ has the following form:
$$\left\{\begin{array}{lll}
[e,h]=-[h,e]=2e, & [h,f]=-[f,h]=2f, & [e,f]=-[f,e]=h,\\[1mm]
[y_1,{y_2}]=-[{y_2},y_1]=y_1,& [x_k,{y_2}]=a x_k, & 0\leq k\leq m, \  a\in \mathbb{C},\\[1mm]
[x_k,h]=(m-2k)x_k & 0\leq k\leq m,&\\[1mm]
[x_k,f]=x_{k+1},  & 0\leq k\leq m-1, & \\[1mm]
[x_k,e]=-k(m+1-k)x_{k-1}, & 1\leq k\leq m. &
\end{array}\right.$$
\end{thm}

The following theorem extends Theorem \ref{thm01} for $dim R=3$.
\begin{thm} \label{thm02} \cite{Rakh1}
Let $L$ be a Leibniz algebra whose quotient
$L/I\cong sl_2\oplus R$, where $R$ is a three-dimensional solvable ideal and  $I$ is an irreducible right module over $sl_2\ (dim I\neq3)$. Then there exists a basis $\{e, h, f,
x_0, x_1, \dots, x_m, y_1, {y_2}, y_3\}$ of the algebra $L$ such that the table of multiplication in $L$ has one of the following two forms:
$$\begin{array}{c}
L_1(\alpha,a):\left\{\begin{array}{lll}[e,h]=-[h,e]=2e, & [h,f]=-[f,h]=2f, & [e,f]=-[f,e]=h,\\[1mm]
[y_1,y_2]=-[y_2,y_1]=y_1, & [y_3,y_2]=-[y_2,y_3]=\alpha y_3, & \\[1mm]
[x_k,h]=(m-2k)x_k, & 0 \leq k \leq m,&\\[1mm]
[x_k,f]=x_{k+1},  & 0 \leq k \leq m-1,& \\[1mm]
[x_k,e]=-k(m+1-k)x_{k-1}, & 1 \leq k \leq m,&\\[1mm]
[x_i,y_2]=a x_i, & 0 \leq i \leq m. \\[1mm]
 \end{array}\right.\\[20mm]
L_2(a): \left\{\begin{array}{lll}
[e,h]=-[h,e]=2e, & [h,f]=-[f,h]=2f, & [e,f]=-[f,e]=h,\\[1mm]
[y_1,y_2]=-[y_2,y_1]=y_1+ y_3, & [y_3,y_2]=-[y_2,y_3]=y_3,&\\[1mm]
[x_k,h]=(m-2k)x_k, & 0 \leq k \leq m,&\\[1mm]
[x_k,f]=x_{k+1},  & 0 \leq k \leq m-1, & \\[1mm]
[x_k,e]=-k(m+1-k)x_{k-1}, & 1 \leq k \leq m, &\\[1mm]
[x_i,y_2]=a x_i, & 0 \leq i \leq m.&
 \end{array}\right.
 \end{array}$$
\end{thm}

For a semisimple Lie algebra $S$ we consider a semisimple Leibniz
algebra $L$ such that $L=(sl_2\oplus S) \dot{+}I.$ We put $I_1 =
[I,sl_2].$

Let $I_1$ be a reducible over $sl_2.$ Then by Theorem \ref{thm25}
we have a decomposition:
$$I_1 = I_{1,1}\oplus I_{1,2}\oplus \dots \oplus I_{1,p},$$ where
$I_{1,j}$ are irreducible modules over $sl_2$ for every $j$, $1 \leq j \leq p.$

\begin{thm} \label{thm28} \cite{Sam}
Let $dim I_{1,j_1} = dim I_{1,j_2} = \cdots=dim I_{1,j_s} =t+1,$ $1\leq s\leq p.$ Then there exist $t+1$ submodules $I_{2,1},
I_{2,2}, \dots I_{2,t+1}$ of dimension $s$ of the module $I_2=[I,
S]$ such that
$$I_{2,1}+ I_{2,2}+ \dots +I_{2,t+1} = I_1 \cap I_2.$$
\end{thm}


\section{On Leibniz algebras whose associated Lie algebra is isomorphic to
$sl_2^1\oplus sl_2^2\oplus \dots \oplus sl_2^s\oplus R.$}

In this section, we will consider a Leibniz algebra satisfying the following conditions:
\begin{enumerate}
\item[$(i)$] the quotient algebra $L/I$ is isomorphic to the direct sum $sl_2^1\oplus sl_2^2\oplus \dots \oplus sl_2^s\oplus R,$
where $R$ is $n$-dimensional solvable Lie algebra;

\item[$(ii)$] the ideal $I$ is a right irreducible $sl_2^k$-module for some $k\in\{1, \dots, s\}.$
\end{enumerate}

By reordering direct sums and changing indexes in the condition
$(ii)$ we can assume $k=1$.

We denote by $<e_i,f_i,h_i>, \ 1\leq i\leq s$ the basis elements
of $L$ which are preimage of standard basis of $sl_2^i$ in the
homomorphism $L \rightarrow L/I$ and we put
$$I=<x_0, \dots, x_m>, \quad R=<y_1, \dots, y_n>.$$

 Then due to \cite{Jac} we have
$$\begin{array}{lll}
\, [e_1,h_1]=-[h_1,e_1]=2e_1, & [h_1,f_1]=-[f_1,h_1]=2f_1, &[e_1,f_1]=-[f_1,e_1]=h_1, \\
\, [x_i,h_1]=(m-2i)x_i, & 0 \leq i \leq m ,&\\
\, [x_i,f_1]=x_{i+1},  & 0 \leq i \leq m-1, & \\
\, [x_i,e_1]=-i(m+1-i)x_{i-1}, & 1 \leq i \leq m. &\\
 \end{array}$$

The correctness of the next lemma follows from Lemma 3.3 in \cite {Sam}.
\begin{lem} \label{lem1} Let $L$ be a Leibniz algebra satisfying the conditions (i)-(ii). Then $[I,sl_2^j]=0$ for any $j\in\{2, \dots, s\}.$
\end{lem}

The following lemma trivially follows from Theorem \ref{t2}.
\begin{lem} \label{lem2} Let $L$ be a Leibniz algebra satisfying the conditions (i)-(ii). Then $[sl_2^t,sl_2^t]=sl_2^t$ for $ 1\leq t\leq s$ and $[sl_2^i,sl_2^j]=0$ for $1\leq i, j\leq s, \ i\neq j.$
\end{lem}

The next lemma establish that the solvable ideal $R$ is contained in two sided annihilator of each $sl_2^i, \ 2\leq i\leq s.$

\begin{lem} \label{lem4} Let $L$ be a Leibniz algebra satisfying the conditions (i)-(ii). Then $$[R,sl_2^i]=[sl_2^i,R]=0, \quad 2\leq i\leq s.$$
\end{lem}
\begin{proof} Applying Leibniz identity for the following triples
$$(y_s, e_1, a), \ (y_s, f_1, a), \ (a, y_s, e_1), \ (a, y_s, f_1)$$
lead $[y_s,a]=0, \quad [a,y_s]=0, \ 1\leq s\leq n$
for an arbitrary element $a\in sl_2^i, \ 2\leq i\leq s.$
\end{proof}

Summarizing the results of Lemmas \ref{lem1}-\ref{lem4}, we obtain the following theorem.

\begin{thm} \label{thm1} Let $L$ be a finite-dimensional Lebniz algebra satisfying conditions:
\begin{itemize}
\item[$(i)$] $L/I\cong sl_2^1\oplus sl_2^2\oplus \dots \oplus sl_2^s\oplus R,$ where $R$ is an $n$-dimensional solvable Lie algebra;\\

\item[$(ii)$] the ideal $I$ is a right irreducible module over $sl_2^1.$
\end{itemize}

Then, $L\cong((sl_2^1\oplus R)\dotplus I)\oplus sl_2^2\oplus \dots
\oplus sl_2^s.$
\end{thm}

As a result of Theorems \ref{thm01}--\ref{thm02} and \ref{thm1}, we have the following  corollaries.

\begin{cor} Let $L/I \cong sl_2^1\oplus sl_2^2\oplus \dots \oplus sl_2^s\oplus R$ with $dimR=2$ and $dimI\neq3.$ Then $L$ is isomorphic to the following algebra:
$$\left\{\begin{array}{lll}
[e_j,h_j]=-[h_j,e_j]=2e_j, & [h_j,f_j]=-[f_j,h_j]=2f_j,\\[1mm]
[e_j,f_j]=-[f_j,e_j]=h_j, &  1\leq j\leq s,\\[1mm]
[y_1,{y_2}]=-[{y_2},y_1]=y_1,&\\[1mm]
[x_k,h_1]=(m-2k)x_k, & 0\leq k\leq m,&\\[1mm]
[x_k,f_1]=x_{k+1},  & 0\leq k\leq m-1, & \\[1mm]
[x_k,e_1]=-k(m+1-k)x_{k-1}, & 1\leq k\leq m, & \\[1mm]
[x_k,{y_2}]=a x_k, & 0\leq k\leq m, \  a\in \mathbb{C}.&
\end{array}\right.$$
\end{cor}

\begin{cor} Let $L/I \cong sl_2^1\oplus sl_2^2\oplus \dots \oplus sl_2^s\oplus R,$  with $dimR=3$ and $dimI\neq3.$ Then $L$ is isomorphic to the following non-isomorphic algebras:
$$\begin{array}{ll}
L_1(\alpha,a):&\left\{\begin{array}{lll}
[e_j,h_j]=-[h_j,e_j]=2e_j, & [h_j,f_j]=-[f_j,h_j]=2f_j, &\\[1mm]
[e_j,f_j]=-[f_j,e_j]=h_j, & 1\leq j\leq s,\\[1mm]
[y_1,y_2]=-[y_2,y_1]=y_1, & [y_3,y_2]=-[y_2,y_3]=\alpha y_3, & \\[1mm]
[x_k,h_1]=(m-2k)x_k, & 0 \leq k \leq m,&\\[1mm]
[x_k,f_1]=x_{k+1},  & 0 \leq k \leq m-1,& \\[1mm]
[x_k,e_1]=-k(m+1-k)x_{k-1}, & 1 \leq k \leq m,&\\[1mm]
[x_i,y_2]=a x_i, & 0 \leq i \leq m, \\[1mm]
 \end{array}\right.\\[25mm]
L_2(a):& \left\{\begin{array}{lll}
[e_j,h_j]=-[h_j,e_j]=2e_j, & [h_j,f_j]=-[f_j,h_j]=2f_j, & \\[1mm]
[e_j,f_j]=-[f_j,e_j]=h_j, & 1\leq j\leq s,\\[1mm]
[y_1,y_2]=-[y_2,y_1]=y_1+ y_3, & [y_3,y_2]=-[y_2,y_3]=y_3,&\\[1mm]
[x_k,h_1]=(m-2k)x_k, & 0 \leq k \leq m,&\\[1mm]
[x_k,f_1]=x_{k+1},  & 0 \leq k \leq m-1, & \\[1mm]
[x_k,e_1]=-k(m+1-k)x_{k-1}, & 1 \leq k \leq m, &\\[1mm]
[x_i,y_2]=a x_i, & 0 \leq i \leq m.&
 \end{array}\right.
 \end{array}$$
\end{cor}

\section{On Leibniz algebras with semisimple part $sl_2^1\oplus sl_2^2.$}

Let the quotient Lie algebra $L/I$ for a Leibniz algebra $L$ be
isomorphic to a direct sum of two copies of the $sl_2$, i.e.,
$L/I\cong sl_2\oplus sl_2$. In this section we shall investigate
the case when the ideal $I$ is reducible over each copy of
$sl_2.$ In order to distinguish copies of $sl_2$ we shall denote
them by $sl_2^1$ and $sl_2^2$. One can assume that $I$ is
reducible over $sl_2^1$. Due to Theorem \ref{thm25} we have the
following decomposition:
$$I=I_{1,1}\oplus I_{1,2}\oplus...\oplus I_{1,s+1},$$ where $I_{1,j}, \ 1\leq j\leq s+1$ are irreducible $sl_2^1$-modules.

We shall focus our study on the case when $dimI_{1,1}=dimI_{1,2}=\dots =dimI_{1,s+1}=m+1.$

Let us introduce notations as follows:
$$I_{1,j}=<x_0^j, x_1^j, \dots, x_m^j>, \ 1\leq j \leq s+1.$$

In the proof of Theorem 3.7 of the paper \cite{Sam} it was proved that

$$[x_i^j,e_2]=\sum\limits_{k=1}^{s+1}a_{i,j}^kx_i^k, \quad [x_i^j,f_2]=\sum\limits_{k=1}^{s+1}b_{i,j}^kx_i^k, \quad [x_i^j,h_2]=\sum\limits_{k=1}^{s+1}c_{i,j}^kx_i^k,$$
where $0\leq i\leq m, \ 1\leq j\leq s+1.$

Without loss of generality, one can assume that the products $[I_{1,j},sl_2^1], \ 1\leq j\leq s+1$ are expressed as follows:
$$\begin{array}{lll}
\, [x_i^j,e_1]=-i(m+1-i)x_{i-1}^j, & 1 \leq i \leq m, \\
\, [x_i^j,f_1]=x_{i+1}^j,  & 0 \leq i \leq m-1, \\
\, [x_i^j,h_1]=(m-2i)x_i^j, & 0 \leq i \leq m.\\
\end{array}$$

\begin{prop} \label{prop1} Let $L/I$ be isomorphic to $sl_2^1\oplus sl_2^2,$ where $I=I_{1,1}\oplus I_{1,2}\oplus...\oplus I_{1,s+1}$ with $dimI_{1,j}=m+1$ and $I_{1,j}$ are irreducible $sl_2^1$-modules for $1\leq j\leq s+1$. Then
$$\begin{array}{lll}
\, [x_i^j,e_2]=\sum\limits_{k=1}^{s+1}a_j^kx_i^k, & [x_i^j,f_2]=\sum\limits_{k=1}^{s+1}b_j^kx_i^k, & [x_i^j,h_2]=\sum\limits_{k=1}^{s+1}c_j^kx_i^k, \\ \end{array}$$
where $0\leq i\leq m, \ 1\leq j\leq s+1.$
\end{prop}
\begin{proof}

Applying Leibniz identity for the following triples of elements:
$$ (x_1^j, e_1, e_2), \ 1\leq j\leq s+1$$
we derive the restrictions:
$$ a_{1,j}^k=a_{0,j}^k,  \ \ \ 1\leq k\leq s+1.$$

Consequently, we obtain
$$[x_0^j,e_2]=\sum\limits_{k=1}^{s+1}a_{0,j}^kx_0^k, \quad [x_1^j,e_2]=\sum\limits_{k=1}^{s+1}a_{0,j}^kx_1^k, \quad 1\leq j\leq s+1.$$

By induction, we shall prove the equality
\begin{equation} \label{eq1}
[x_i^j,e_2]=\sum\limits_{k=1}^{s+1}a_{0,j}^kx_i^k, \quad 0\leq i\leq m.
\end{equation}

Using the assumption of correctness of Equality \ref{eq1} for $i$ in the following chain of equalities:
$$\begin{array}{l}
0=[x_{i+1}^j,[e_1,e_2]]=[[x_{i+1}^j,e_1],e_2]-[[x_{i+1}^j,e_2],e_1]=-[(i+1)(m-i)x_i^j,e_2]-\\[2mm]
-\sum\limits_{k=1}^{s+1}a_{i+1,j}^k[x_{i+1}^k,e_1]=
-(i+1)(m-i)\sum\limits_{k=1}^{s+1}a_{0,j}^kx_i^k+
\sum\limits_{k=1}^{s+1}a_{i+1,j}^k(i+1)(m-i)x_i^k,
\end{array}$$
we conclude that $a_{i+1,j}^k=a_{0,j}^k$ for $1\leq k\leq s+1$, that is,
$[x_{i+1}^j,e_2]=\sum\limits_{k=1}^{s+1}a_{0,j}^kx_{i+1}^k$ and Equality \ref{eq1} is proved.


Putting $a_j^k=a_{0,j}^k,$ we have $[x_i^j,e_2]=\sum\limits_{k=1}^{s+1}a_j^kx_i^k, \ 1\leq j\leq s+1, \ 0\leq i\leq m.$

Applying Leibniz identity for the triples of elements: $$(x_1^j, e_1, f_2), \quad 1\leq j\leq s+1,$$
we get
$$b_{1,j}^k=b_{0,j}^k,  \ 1\leq k\leq s+1.$$

Therefore, we obtain
$$[x_0^j,f_2]=\sum\limits_{k=1}^{s+1}b_{0,j}^kx_0^k, \quad [x_1^j,f_2]=\sum\limits_{k=1}^{s+1}b_{0,j}^kx_1^k, \ 1\leq j\leq s+1.$$

Applying induction and the following chain of equalities

$$\begin{array}{l}
0=[x_{i+1}^j,[e_1,f_2]]=[[x_{i+1}^j,e_1],f_2]-[[x_{i+1}^j,f_2],e_1]=-[(i+1)(m-i)x_i^j,f_2]-\\[2mm]
-\sum\limits_{k=1}^{s+1}b_{i+1,j}^k[x_{i+1}^k,e_1]=
-(i+1)(m-i)\sum\limits_{k=1}^{s+1}b_{0,j}^kx_i^k+
\sum\limits_{k=1}^{s+1}b_{i+1,j}^k(i+1)(m-i)x_i^k,
\end{array}$$

we derive the equality
$$[x_i^j,f_2]=\sum\limits_{k=1}^{s+1}b_{0,j}^kx_i^k, \ 0\leq i\leq m, \ 1\leq j\leq s+1.$$

Setting $b_j^k=b_{0,j}^k,$ we obtain
$[x_i^j,f_2]=\sum\limits_{k=1}^{s+1}b_j^kx_i^k, \ 0\leq i\leq m, \ 1\leq j\leq s+1.$

Analogously, one can prove the equality
$[x_i^j,h_2]=\sum\limits_{k=1}^{s+1}c_j^kx_i^k$ with $ 1\leq j\leq s+1.$
\end{proof}

Now we shall describe Leibniz algebras such that $L/I\cong sl_2^1\oplus sl_2^2$ and $I=I_{1,1}\oplus I_{1,2},$ where $I_{1,1}, I_{1,2}$ are irreducible $sl_2^1$-modules.
Without loss of generality we can suppose
$$\begin{array}{lll}
\, [x_i^j,h_1]=(m-2i)x_i^j, & 0 \leq i \leq m,\\
\, [x_i^j,f_1]=x_{i+1}^j,  & 0 \leq i \leq m-1,\\
\, [x_i^j,e_1]=-i(m+1-i)x_{i-1}^j, & 1 \leq i \leq m.\\
 \end{array}$$
for $j=1, 2.$

According to Proposition \ref{prop1}, one can assume that
$$\begin{array}{lll}
\, [x_i^1,e_2]=a_1x_i^1+a_2x_i^2,& [x_i^2,e_2]=a_3x_i^1+a_4x_i^2,\\
\, [x_i^1,f_2]=b_1x_i^1+b_2x_i^2,& [x_i^2,f_2]=b_3x_i^1+b_4x_i^2,\\
\, [x_i^1,h_2]=c_1x_i^1+c_2x_i^2,& [x_i^2,h_2]=c_3x_i^1+c_4x_i^2,\\
 \end{array}$$
where  $0\leq i\leq m.$

From the following chains of the equalities obtained applying Leibniz identity
$$\begin{array}{l}
2(a_1x_0^1+a_2x_0^2)=2[x_0^1,e_2]=[x_0^1,[e_2,h_2]]=(a_2c_3-c_2a_3)x_0^1+(a_1c_2+a_2c_4-c_1a_2-c_2a_4)x_0^2,\\[2mm]
2(a_3x_0^1+a_4x_0^2)=2[x_0^2,e_2]=[x_0^2,[e_2,h_2]]=(a_3c_1+a_4c_3-c_3a_1-c_4a_3)x_0^1+(a_3c_2-a_2c_3)x_0^2,\\[2mm]
-2(b_1x_0^1+b_2x_0^2)=-2[x_0^1,f_2]=[x_0^1,[f_2,h_2]]=(b_2c_3-c_2b_3)x_0^1+(b_1c_2+b_2c_4-c_1b_2-c_2b_4)x_0^2,\\[2mm]
-2(b_3x_0^1+b_4x_0^2)=2[x_0^2,e_2]=[x_0^2,[f_2,h_2]]=(b_3c_1+b_4c_3-c_3b_1-c_4b_3)x_0^1+(b_3c_2-b_2c_3)x_0^2,\\[2mm]
-c_1x_1^1-c_2x_1^2=-[x_1^1,h_2]=[x_1^1,[f_2,e_2]]=(a_3b_2-a_2b_3)x_1^1+2(a_2b_1-a_1b_2)x_1^2,\\[2mm]
-c_3x_1^1-c_4x_1^2=-[x_1^2,h_2]=[x_1^2,[f_2,e_2]]=2(a_1b_3-a_3b_1)x_1^1+(a_2b_3-a_3b_2)x_1^2.
\end{array}$$
we derive:
\begin{equation}\label{eq777}
\left\{\begin{array}{l}
2a_1=a_2c_3-a_3c_2,  \\[1mm]
2a_2=a_1c_2+a_2c_4-c_1a_2-c_2a_4, \\[1mm]
2a_3=a_3c_1+a_4c_3-c_3a_1-c_4a_3,  \\[1mm]
2a_4=a_3c_2-a_2c_3,\\[1mm]
-2b_1=b_2c_3-c_2b_3,  \\[1mm]
-2b_2=b_1c_2+b_2c_4-c_1b_2-c_2b_4,\\[1mm]
-2b_3=b_3c_1+b_4c_3-c_3b_1-c_4b_3,  \\[1mm]
-2b_4=b_3c_2-b_2c_3,\\[1mm]
c_1=a_2b_3-a_3b_2,  \\[1mm]
c_2=2(a_1b_2-a_2b_1), \\[1mm]
c_3=2(a_3b_1-a_1b_3), \\[1mm]
c_4=a_3b_2-a_2b_3.
\end{array} \right.
\end{equation}
It is easy to see that $a_4=-a_1, b_4=-b_1$ and $c_4=-c_1.$

Thus, we obtain the following products:
\begin{equation} \label{eq333}
\begin{array}{ll}
\, [x_i^1,e_2]=a_1x_i^1+a_2x_i^2,& [x_i^1,f_2]=b_1x_i^1+b_2x_i^2,\\[1mm]
\, [x_i^2,e_2]=a_3x_i^1-a_1x_i^2,& [x_i^2,f_2]=b_3x_i^1-b_1x_i^2,\\[1mm]
\, [x_i^1,h_2]=(a_2b_3-a_3b_2)x_i^1+2(a_1b_2-a_2b_1)x_i^2,& \\[1mm]
\, [x_i^2,h_2]=2(a_3b_1-a_1b_3)x_i^1-(a_2b_3-a_3b_2)x_i^2,&
 \end{array}
 \end{equation}
where the structure constants $a_1, a_2, a_3$ and $b_1, b_2, b_3$ satisfy the relations (\ref{eq777}).

We present the classification of Leibniz algebras satisfying the following conditions:
\begin{itemize}
\item[$(a)$] $L/I\cong sl_2^1\oplus sl_2^2;$

\item[$(b)$] $I=I_{1,1}\oplus I_{1,2}$ such that $I_{1,1}, I_{1,2}$ are irreducible $sl_2^1$-modules and  $dimI_{1,1}=dimI_{1,2};$

\item[$(c)$] $I=I_{2,1}\oplus I_{2,2}\oplus...\oplus I_{2,m+1}$ such that $I_{2,k}$ are irreducible $sl_2^2$-modules with $1\leq k\leq m+1$.
\end{itemize}

\begin{thm} \label{thm2} An arbitrary Leibniz algebra satisfying the conditions (a)-(c) is isomorphic to the following algebra:
$$\left\{\begin{array}{lll}
\, [e_i,h_i]=-[h_i,e_i]=2e_i, & \\
\, [e_i,f_i]=-[f_i,e_i]=h_i, &   & \\
\, [h_i,f_i]=-[f_i,h_i]=2f_i, & \\
\, [x_k^i,h_1]=(m-2k)x_k^i, & 0 \leq k \leq m ,\\
\, [x_k^i,f_1]=x_{k+1}^i,  & 0 \leq k \leq m-1, \\
\, [x_k^i,e_1]=-k(m+1-k)x_{k-1}^i, & 1 \leq k \leq m, \\
\, [x_j^1,e_2]=[x_j^2,h_2]=x_j^2,&\\
\, [x_j^1,h_2]=[x_j^2,f_2]=-x_j^1,&
 \end{array}\right.$$
with $1\leq i\leq2$ and $ 0\leq j\leq m.$
\end{thm}
\begin{proof}
We set $dimI_{1,1}=dimI_{1,2}=m+1.$ Then, according to Theorem \ref{thm28}, we obtain $dim I_{2,k}=2$ for $1\leq k\leq m+1.$

Let $\{x_0^1,x_1^1,...,x_m^1\}, \ \{x_0^2,x_1^2,...,x_m^2\}$ and $\{y_{i+1}^1,y_{i+1}^2\}$ be bases of $I_{1,1}, \ I_{1,2}$ and $I_{2,i+1}, \  0\leq i\leq m$, respectively.
We set
\begin{equation} \label{eq4444}
y_{i+1}^1=\alpha_{i}^1x_{i}^1+\alpha_{i}^2x_{i}^2, \quad y_{i+1}^2=\beta_{i}^1x_{i}^1+\beta_{i}^2x_{i}^2, \quad 0\leq i\leq m.
\end{equation}

Taking into account the products (\ref{eq333}) for $0\leq i\leq m$ we consider the equalities
$$\begin{array}{l}
0=[y_{i+1}^2,f_2]=[\beta_i^1x_i^1+
\beta_i^2x_i^2,f_2]=\beta_i^1(b_1x_i^1+b_2x_i^2)+\\
+\beta_i^2(b_3x_i^1-b_1x_i^2)=
(\beta_i^1b_1+\beta_i^2b_3)x_i^1+
(\beta_i^1b_2-\beta_i^2b_1)x_i^2.
\end{array}$$

Therefore,
\begin{equation} \label{eq444}
\left\{\begin{array}{l}
\beta_i^1b_1+\beta_i^2b_3=0,  \\[1mm]
\beta_i^1b_2-\beta_i^2b_1=0,
\end{array} \right.
\end{equation}
with $0\leq i\leq m.$

If $b_1^2+b_2b_3\neq0,$ then the system of equations (\ref{eq444}) has only trivial solution, which is a contradiction. Hence, $b_1^2+b_2b_3=0.$

Similarly, from
$$0=[y_{i+1}^1,e_2]=(\alpha_i^1a_1+\alpha_i^2a_3)x_i^1+
(\alpha_i^1a_2-\alpha_i^2a_1)x_i^2$$
we derive $a_1^2+a_2a_3=0.$

Thus, we have $a_1=\pm i\sqrt{a_2a_3}$ and $b_1=\pm
i\sqrt{b_2b_3}.$

Let us summarize the obtained products:
\begin{equation} \label{eq666}
\begin{array}{ll}
\, [x_i^1,e_2]=a_1x_i^1+a_2x_i^2,& [x_i^1,f_2]=b_1x_i^1+b_2x_i^2,\\
\, [x_i^2,e_2]=a_3x_i^1-a_1x_i^2, & [x_i^2,f_2]=b_3x_i^1-b_1x_i^2,\\
\, [x_i^1,h_2]=(a_2b_3-a_3b_2)x_i^1+2(a_1b_2-a_2b_1)x_i^2,&\\
\, [x_i^2,h_2]=2(a_3b_1-a_1b_3)x_i^1-(a_2b_3-a_3b_2)x_i^2,&\\
 \end{array}
\end{equation} with $0\leq i\leq m$ and the relations $a_1^2+a_2a_3=b_1^2+b_2b_3=0.$

Taking the following basis transformation:
$$x_i^{1'}=Ax_i^1+Bx_i^2, \ x_i^{2'}=(Aa_1+Ba_3)x_i^1+(Aa_2-Ba_1)x_i^2, \ 0\leq i\leq m$$
we can assume that the products (\ref{eq666}) have the following form:
$$\begin{array}{ll}
\, [x_i^1,e_2]=x_i^2,& [x_i^2,e_2]=0,\\
\, [x_i^1,f_2]=b_1x_i^1+b_1^2x_i^2,& [x_i^2,f_2]=-x_i^1-b_1x_i^2,\\
\, [x_i^1,h_2]=-x_i^1-2b_1x_i^2,& [x_i^2,h_2]=x_i^2.\\
 \end{array}$$

Applying the change of basis as follows $$x_i^{1'}=x_i^1+b_1x_i^2, \ x_i^{2'}=x_i^2, \ \ 0\leq i\leq m,$$
we complete the proof of the theorem.
\end{proof}

The following theorem establishes that condition (c) can be
omitted because if conditions (a)-(b) are true, then condition
(c) is always executable.

\begin{thm} Let $L$ be a Leibniz algebra satisfying the conditions (a)-(b). Then either $L$ satisfies the condition (c) or $L\cong (sl_2^1+I)\oplus sl_2^2$.
\end{thm}

\begin{proof} Let $L$ be a Leibniz algebra satisfing conditions $(a)$ and $(b),$ but not
$(c)$. In order to prove the assertion of theorem we have to establish that all modules $I_{2,i}, \ 1 \leq i \leq m+1$ are reducible over $sl_2^2.$ Indeed, according to Theorem \ref{thm25} we conclude that $I_{2,i}$ are
completely reducible modules over $sl_2^2$. In denotation of (\ref{eq4444}) we have $I_{2,i}=<y_i^1>\oplus<y_i^2>,$ where $<y_i^1>, \ <y_i^2>$ are one-dimensional trivial $sl_2^2$-modules, that
is, $$[y_i^1,e_2]=[y_i^2,e_2]=[y_i^1,f_2]=[y_i^2,f_2]=[y_i^1,h_2]=[y_i^2,h_2]=0.$$

We shall prove by contrary method, that is, we shall assume that not all modules $I_{2,i}$ are reducible over $sl_2^2.$ Then we can assume that there exist some $s, \ 1 \leq s \leq m+1$ and
$t, \ 1\leq t\leq m+1, \ t\neq s$ such that $I_{2,s}$ is irreducible, but $I_{2,t}$ is reducible $sl_2^2$-modules. By renumerating of indexes, without loss of generality, we can suppose $s=2$ and
$t=1.$

From the products in the proof of Theorem \ref{thm2} we have
$$[x_{1}^1,e_2]=[x_{1}^2,h_2]=x_{1}^2,\quad [x_{1}^1,h_2]=[x_{1}^2,f_2]=-x_{1}^1, \quad [x_{1}^2, e_2]=[x_{1}^1,f_2]=0.$$

Consider the chain of equalities
$$0=[y_1^1,[e_2,f_1]]=[[y_1^1,e_2],f_1]-[[y_1^1,f_1],e_2]=-[[y_1^1,f_1],e_2]=$$
$$-[[\alpha_{0}^1x_{0}^1+\alpha_{0}^2x_{0}^2,f_1],e_2]=
-[\alpha_{0}^1x_{1}^1+\alpha_{0}^2x_{1}^2,e_2]=-\alpha_{0}^1x_1^2,$$

$$0=[y_1^1,[f_2,f_1]]=[[y_1^1,f_2],f_1]-[[y_1^1,f_1],f_2]=-[[y_1^1,f_1],f_2]=$$
$$-[[\alpha_{0}^1x_{0}^1+\alpha_{0}^2x_{0}^2,f_1],e_2]=-
[\alpha_{0}^1x_{1}^1+\alpha_{0}^2x_{1}^2,f_2]=\alpha_{0}^2x_1^1.$$

Therefore, $\alpha_{0}^1=\alpha_{0}^2=0,$ which means $y_1^1=0$. Thus, we get a contradiction.
\end{proof}

\section{Some remarks on Leibniz algebras with semisimple part $sl_n.$}

In this section we present the structure of Leibniz algebras with associated Lie algebra $sl_n$ and with dimension of ideal $I$ equal to 1 and 2.

\begin{prop} Let $L/I$ be isomorphic to $G$ where $G=<e_1,e_2,...,e_m>$ is a semisimple Lie algebra and $dimI=1.$ Then
$L=G\oplus I.$
\end{prop}
\begin{proof} We put $I=<x>$ and $[x,e_i]=\alpha_ix.$ For the semisimple Lie algebra $G$ we have $[G,G]=G$, that is, for any $e_i, \ 1\leq i \leq m$ there exist $e_p$ and $e_q$ such that $e_i=\sum\limits_{p,q}\beta_{p,q}[e_p,e_q].$

The proof of proposition completes the following chain of equalities:
$$[x,e_i]=[x,\sum\limits_{p,q}\beta_{p,q}[e_p,e_q]]=\sum\limits_{p,q}\beta_{p,q}[x,[e_p,e_q]]=
\sum\limits_{p,q}\beta_{p,q}([[x,e_p],e_q]-[[x,e_q],e_p])=$$
$$\sum\limits_{p,q}\beta_{p,q}(\alpha_p[x,e_q]-\alpha_q[x,e_p])=\sum\limits_{p,q}\beta_{p,q}(\alpha_p\alpha_qx-\alpha_q\alpha_px)=0.$$
\end{proof}

Let $G$ be a simple Lie algebra with a basis $\{e_1, e_2, \dots e_m\}$ which satisfies the condition that for any $e_i$ there exist $p,q$ such that $e_i=[e_p,e_q].$

We consider the case of $L/I\cong G$ and $dimI=2.$
We set $\{x,y\}$ the basis elements of $I$ and
$$[x,e_i]=\alpha_ix+\beta_iy, \quad [y,e_i]=\gamma_ix+\delta_iy, \quad 1\leq i\leq n.$$

Consider
$$[x,e_i]=[x,[e_p,e_q]]=[[x,e_p],e_q]-[[x,e_q],e_p]=(\beta_p\gamma_q-\beta_q\gamma_p)x+(\alpha_p\beta_q+\beta_p\delta_q-\alpha_q\beta_p-\beta_q\delta_p)y.$$

On the other hand, $[x,e_i]=\alpha_ix+\beta_iy.$

Comparing the coefficients we derive
\begin{equation}\label{eq5}
\left\{\begin{array}{l} \alpha_i=\beta_p\gamma_q-\beta_q\gamma_p,  \\[1mm] \beta_i=\alpha_p\beta_q+\beta_p\delta_q-\alpha_q\beta_p-\beta_q\delta_p, \end{array} \right.
\end{equation}
where $1\leq i\leq n.$

Similarly, if we consider $[y,e_i],$ we deduce

\begin{equation} \label{eq55}\left\{\begin{array}{l} \gamma_i=\gamma_p\alpha_q+\delta_p\gamma_q-\gamma_q\alpha_p-\delta_q\gamma_p,  \\[1mm] \delta_i=\gamma_p\beta_q-\gamma_q\beta_p, \end{array} \right.
\end{equation}
where $1\leq i\leq n.$

From systems (\ref{eq5}) and (\ref{eq55}) we obtain
\begin{equation}\label{eq555}
[x,e_i]=\alpha_ix+\beta_iy, \quad [y,e_i]=\gamma_ix-\alpha_iy, \quad 1\leq i\leq n.
\end{equation}

Let $L/I$ be isomorphic to $sl_n$. From \cite{Jac} we have the standard basis  of $sl_n$ $<h_k,e_{i,j}>$ with $1\leq k\leq n-1,$ $ 1\leq i,j\leq n$ and $ i\neq j$. We recall the table of multiplication of $sl_n$
$$\begin{array}{llll}
[e_{ij},e_{jk}]=e_{ik},& &	& 1\leq i,j,k\leq n,   i\neq j, \quad j\neq k, \quad k\neq i,\\{}
[h_k,e_{in}]=e_{in},	& [h_k,e_{ni}]=-e_{ni},	& & 1\leq i,k\leq n-1, \quad k\neq i,\\{}
[h_i,e_{ij}]=e_{ij},	&  [h_j,e_{ij}]=-e_{ij},	& [e_{ij},e_{ji}]=h_i-h_j, & 1\leq i,j\leq n-1,\\{}
[h_i,e_{in}]=2e_{in},	& [h_i,e_{ni}]=-2e_{ni}, & [e_{in},e_{ni}]=h_i , & 1\leq i\leq n-1.
\end{array}$$

\begin{thm} Let $L/I$ be isomorphic to $sl_n\ (n\geq3),$ where $dimI=2.$ Then $L=sl_n\oplus I.$
\end{thm}
\begin{proof} From (\ref{eq555}) we have
$$[x,h_i]=\alpha_ix+\beta_iy, 1\leq i\leq n-1,\quad [y,h_i]=\gamma_ix-\alpha_iy, \quad 1\leq i\leq n-1,$$
$$[x,e_{ij}]=\alpha_{ij}x+\beta_{ij}y, \quad [y,e_{ij}]=\gamma_{ij}x-\alpha_{ij}y, \quad 1\leq i,j\leq n, \quad i\neq j.$$

Applying Leibniz identity for the following triples of elements
$$(x, h_i, e_{ij}), \quad (y, h_i, e_{ij}), \quad (x,h_i,e_{in}), \quad (y,h_i,e_{in}), \quad (x,h_i,e_{ni}), \quad (y,h_i,e_{ni}).$$
we deduce
\begin{equation}\label{eq6}
\left\{\begin{array}{l} \alpha_{ij}+\gamma_i\beta_{ij}-\beta_i\gamma_{ij}=0,  \\[1mm] 2\beta_i\alpha_{ij}+(1-2\alpha_i)\beta_{ij}=0, \\[1mm] -2\gamma_i\alpha_{ij}+(1+2\alpha_i)\gamma_{ij}=0, \end{array} \right.
\end{equation}
\begin{equation}\label{eq66}\left\{\begin{array}{l} 2\alpha_{in}+\gamma_i\beta_{in}-\beta_i\gamma_{in}=0,  \\[1mm] \beta_i\alpha_{in}+(1-\alpha_i)\beta_{in}=0, \\[1mm] -\gamma_i\alpha_{in}+(1+\alpha_i)\gamma_{in}=0, \end{array} \right.
\end{equation}
\begin{equation}\label{eq666}\left\{\begin{array}{l} -2\alpha_{ni}+\gamma_i\beta_{ni}-\beta_i\gamma_{ni}=0,  \\[1mm] \beta_i\alpha_{ni}-(\alpha_i+1)\beta_{ni}=0, \\[1mm]
-\gamma_i\alpha_{ni}+(\alpha_i-1)\gamma_{ni}=0.
\end{array} \right.
\end{equation}

Determinants of systems (\ref{eq6}) - (\ref{eq666}) have the following values:
$$\begin{array}{l}
Det\left( \begin{array}{ccc} 1 & \gamma_i & -\beta_i \\
2\beta_i & 1-2\alpha_i & 0 \\
-2\gamma_i & 0 & 1+2\alpha_i \\
 \end{array} \right)=1-4(\alpha_i^2+\beta_i\gamma_i),\\[5mm]
Det\left( \begin{array}{ccc} 2 & \gamma_i & -\beta_i \\
\beta_i & 1-\alpha_i & 0 \\ -\gamma_i & 0 & 1+\alpha_i \\
\end{array} \right)=2-2(\alpha_i^2+\beta_i\gamma_i),\\[5mm]
Det\left( \begin{array}{ccc} -2 & \gamma_i & -\beta_i \\ \beta_i & -1-\alpha_i & 0 \\ -\gamma_i & 0 & \alpha_i-1 \\ \end{array} \right)=-2+2(\alpha_i^2+\beta_i\gamma_i).
\end{array}$$

It is easy to see that these determinants do not equal to zero simultaneously. Therefore, we conclude that
either $$\alpha_{in}=\beta_{in}=\gamma_{in}=\alpha_{ni}=\beta_{ni}=\gamma_{ni}=0, \ 1\leq i\leq n-1$$ or $$\alpha_{ij}=\beta_{ij}=\gamma_{ij}=0, \ 1\leq i,j\leq n-1.$$

\noindent \textbf{Case 1.} Let $\alpha_{in}=\beta_{in}=\gamma_{in}=\alpha_{ni}=\beta_{ni}=\gamma_{ni}=0$ be where $1\leq i\leq n-1.$ Then the non-zero products characterized $[I,sl_n]$ are the following
$$[x,h_i]=\alpha_ix+\beta_iy, \quad [y,h_i]=\gamma_ix-\alpha_iy, \quad 1\leq i\leq n-1,$$
$$[x,e_{ij}]=\alpha_{ij}x+\beta_{ij}y, \quad [y,e_{ij}]=\gamma_{ij}x-\alpha_{ij}y, \quad 1\leq i,j\leq n-1, \quad i\neq j.$$

From the equalities
$$[x,h_i]=[x,[e_{in},e_{ni}]]=[[x,e_{in}],e_{ni}]-[[x,e_{ni}],e_{in}]=0,$$
$$[y,h_i]=[y,[e_{in},e_{ni}]]=[[y,e_{in}],e_{ni}]-[[y,e_{ni}],e_{in}]=0,$$
$$[x,e_{ij}]=[x,[e_{in},e_{nj}]]=[[x,e_{in}],e_{nj}]-[[x,e_{nj}],e_{ni}]=0,$$
$$[y,e_{ij}]=[y,[e_{in},e_{nj}]]=[[y,e_{in}],e_{nj}]-[[y,e_{nj}],e_{ni}]=0,$$
we obtain that $[I,sl_n]=0.$

\noindent \textbf{Case 2.} Let $\alpha_{ij}=\beta_{ij}=\gamma_{ij}=0$ be where $1\leq i,j\leq n-1.$ Then we have
$$\begin{array}{lll}
[x,h_i]=\alpha_ix+\beta_iy, & [y,h_i]=\gamma_ix-\alpha_iy, & 1\leq i\leq n-1,\\{}
[x,e_{in}]=\alpha_{in}x+\beta_{in}y, & [y,e_{in}]=\gamma_{in}x-\alpha_{in}y, & 1\leq i\leq n-1,\\{}
[x,e_{ni}]=\alpha_{ni}x+\beta_{ni}y, & [y,e_{ni}]=\gamma_{ni}x-\alpha_{ni}y, & 1\leq i\leq n-1,\\{}
[x,e_{ij}]=[y,e_{ij}]=0, & & 1\leq i,j\leq n-1, \quad i\neq j.
\end{array}$$

Using Leibniz identity for the following triples
$$\begin{array}{llll}
(x,e_{ij},e_{ji}), & (y,e_{ij},e_{ji}), & (x,h_j,e_{in}), & (y,h_j,e_{in}),\\
(x,h_j,e_{ni}), & (y,h_j,e_{ni}), & (x,e_{in},e_{ni}), & (y,e_{in},e_{ni}),
\end{array}$$
we obtain
$$ [x,h_i]=[y,h_i]=[x,e_{in}]=[y,e_{in}]=[x,e_{ni}]=[y,e_{ni}]=0, \quad 1\leq i\leq n-1.$$
\end{proof}

\section*{Acknowledgements} The authors are sincerely grateful to anonymous referees for critical remarks.

\end{document}